\documentclass[11pt]{amsart}

\usepackage{xcolor}
\usepackage{amscd,verbatim}
\usepackage{amsmath, amssymb}
\usepackage{amsfonts}

\newcommand{\de}{\partial}
\newcommand{\db}{\overline{\partial}}

\newcommand{\ddbar}{i \partial \overline{\partial}}

\newcommand{\ov}[1]{\overline{#1}}

\newcommand{\ti}[1]{\tilde{#1}}
\newcommand{\vp}{\varphi}
\newcommand{\vol}{\mathrm{Vol}}

\newcommand{\ve}{\varepsilon}

\renewcommand{\leq}{\leqslant}
\renewcommand{\geq}{\geqslant}

\newcommand{\be}{\begin{equation}}
\newcommand{\ee}{\end{equation}}

\begin{document}
\newcounter{remark}
\newcounter{theor}
\setcounter{remark}{0}
\setcounter{theor}{1}
\newtheorem{claim}{Claim}
\newtheorem{theorem}{Theorem}[section]
\newtheorem{lemma}[theorem]{Lemma}
\newtheorem{corollary}[theorem]{Corollary}
\newtheorem{conjecture}[theorem]{Conjecture}
\newtheorem{proposition}[theorem]{Proposition}
\newtheorem{question}[theorem]{Question}
\theoremstyle{definition}
\newtheorem{defn}[theorem]{Definition}
\theoremstyle{definition}
\newtheorem{rmk}{Remark}[section]

\newenvironment{example}[1][Example]{\addtocounter{remark}{1} \begin{trivlist}
\item[\hskip
\labelsep {\bfseries #1  \thesection.\theremark}]}{\end{trivlist}}

\title{Semipositive line bundles and $(1,1)$-classes}
\author{Valentino Tosatti}
\address{Courant Institute of Mathematical Sciences, New York University, 251 Mercer St, New York, NY 10012}
\email{tosatti@cims.nyu.edu}
\begin{abstract}
We survey some recent developments on various notions of semipositivity for $(1,1)$-classes on complex manifolds, and discuss a number of open questions.
\end{abstract}

\maketitle

\section{Introduction}
Let $X^n$ be a compact K\"ahler manifold. A closed real $(1,1)$-form $\alpha$ on $X$ defines a cohomology class
\begin{equation}\label{quot}
[\alpha]\in H^{1,1}(X,\mathbb{R})=\frac{\{\text{closed real }(1,1)\text{-forms}\}}{\{\ddbar \vp\ |\ \vp\in C^\infty(X,\mathbb{R})\}},
\end{equation}
in the finite-dimensional real vector space $H^{1,1}(X,\mathbb{R})$ of $(1,1)$-classes. Kodaira's $\de\db$-Lemma shows that the natural map
$H^{1,1}(X,\mathbb{R})\to H^2(X,\mathbb{R})$ to deRham cohomology is an injection.

The simplest example of a $(1,1)$-class is $[\omega]$, where $\omega$ is a K\"ahler form on $X$, which we will refer to as a {\em K\"ahler class}.  Every K\"ahler class is nonzero since
$$[\omega]^n\cdot X=\int_X\omega^n=n!\, \vol_g(X)>0,$$
where $g$ is the Hermitian metric associated to $\omega$. If we let $\omega$ vary among all K\"ahler forms on $X$, the corresponding K\"ahler classes $[\omega]$ define a cone
$$\mathcal{C}=\{[\omega]\ |\ \omega \text{ K\"ahler form on }X\}\subset H^{1,1}(X,\mathbb{R}),$$
which is nonempty, open and convex, the K\"ahler cone of $X$, see e.g. \cite[p.288]{To4}. We think of K\"ahler classes as ``strictly positive'' classes.

If $L\to X$ is a holomorphic line bundle and $h$ is a smooth Hermitian metric on $L$, then the curvature form $R_h$ of $h$ is a closed real $(1,1)$-form on $X$ given locally by
$$R_h=-\frac{i}{2\pi}\de\db\log h,$$
and its cohomology class
$$c_1(L)=[R_h]\in H^{1,1}(X,\mathbb{R}),$$
is easily seen to be independent of the choice of $h$, and is called the first Chern class of $L$. Its image in $H^2(X,\mathbb{R})$ lies in $N^1(X)$ (which is defined as the image of the change of scalars map
$H^2(X,\mathbb{Z})\to H^2(X,\mathbb{R})$), and conversely every $(1,1)$-class in $N^1(X)$ is of the form $c_1(L)$ for some holomorphic line bundle $L$.  A line bundle $L$ with $c_1(L)\in\mathcal{C}$ is called ample, and by the Kodaira embedding theorem sections of some positive tensor power of $L$ provide a projective embedding $X\subset\mathbb{P}^N$.

The main object of interest in this survey are $(1,1)$-classes in the closure of the K\"ahler cone $\ov{\mathcal{C}}$, which are known as nef classes. Being limits of K\"ahler classes, nef classes can naively be thought of as ``semipositive'' classes. However, as we will see, this is not quite true literally. The main theme of this note will be to explore various notions of semipositivity for $(1,1)$-classes, as well as line bundles, and to draw relations between them, some proven and some conjectural. In particular, we will discuss the author's Conjectures \ref{c4} and \ref{c0}, as well as Conjecture \ref{c2} of Filip and the author, and the new Question \ref{bzp} and Conjecture \ref{c5}. We will also mention a geometric application of these results to the study of families of Ricci-flat K\"ahler metrics, and in the last section we will consider the more general setting of (possibly non-K\"ahler) compact complex manifolds.

\subsection*{Acknowledgments}The author thanks O. Das for comments on Question \ref{bzp}, M. Sroka for discussions about Section \ref{nk} and D. Kim, X. Yang for very useful comments. The author was partially supported by NSF grant DMS-2231783. It is a pleasure to dedicate this note to the memory of Professor Zhihua Chen, whose work on the Schwarz Lemma in complex geometry (including \cite{CCL,CY}) had a big impact on the author's early work \cite{To0}.

\section{Nef $(1,1)$-classes}
\subsection{Nef and semipositive classes}
As in the Introduction, $X^n$ will be a compact K\"ahler manifold, and $\alpha$ a closed real $(1,1)$-form on $X$. Its cohomology class in $H^{1,1}(X,\mathbb{R})$ is denoted by $[\alpha]$.

Recall that in the introduction we have defined the K\"ahler cone $\mathcal{C}\subset H^{1,1}(X,\mathbb{R})$, consisting of cohomology classes of K\"ahler forms, and defined the nef cone $\mathcal{C}$ as the closure of $\mathcal{C}$.

It is elementary to see (e.g. \cite[Lemma 2.2]{To4}) that a $(1,1)$-class $[\alpha]$ is nef if and only if for every $\ve>0$ there is $\vp_\ve\in C^\infty(X,\mathbb{R})$ such that
\begin{equation}\label{succa}
\alpha+\ddbar\vp_\ve\geq -\ve\omega,
\end{equation}
on $X$. It is immediate to see that this notion does not depend on the choice of reference K\"ahler form $\omega$. We will also say that a nef class $[\alpha]$ is nef and big if it satisfies
$$[\alpha]^n\cdot X=\int_X\alpha^n>0.$$

We will say that a $(1,1)$-class $[\alpha]$ is semipositive if it contains a smooth semipositive representative, namely if there exists $\vp\in C^\infty(X,\mathbb{R})$ such that $$\alpha+\ddbar\vp\geq 0,$$
on $X$. The convex cone of semipositive $(1,1)$-classes will be denoted by $\mathcal{S}\subset H^{1,1}(X,\mathbb{R})$, and we have the obvious inclusions
\begin{equation}\label{cones}
\mathcal{C}\subset\mathcal{S}\subset\ov{\mathcal{C}}.
\end{equation}
Intersecting these cones with $N^1(X)$ we obtain the corresponding notions for line bundles, so $L$ is called nef if $c_1(L)\in\ov{\mathcal{C}}$ and $L$ is called semipositive if $c_1(L)\in\mathcal{S}$. Explicitly, a holomorphic line bundle $L\to X$ is nef if
for every $\ve>0$ it admits a smooth Hermitian metric $h_\ve$ whose curvature satisfies
$$R_{h_\ve}\geq -\ve\omega,$$
on $X$, and $L$ is  semipositive if it admits a smooth Hermitian metric $h$ with
$$R_h\geq 0.$$
When $X$ is projective (or more generally Moishezon \cite{Pau}), a line bundle is nef if and only if $(L\cdot C)\geq 0$ for all curves $C\subset X$, and the latter is the usual definition in algebraic geometry.
The notion of a semipositive line bundle is classical, see e.g. \cite{KO, Fu}. Both Kobayashi-Ochiai \cite[P.519]{KO} and Fujita \cite[Conjecture 4.14]{Fu} asked whether on a projective manifold $X$ every nef line bundle is semipositive. This conjecture turned out to be false, as shown by Yau \cite[P.228]{Y1}:
\begin{theorem}[Yau \cite{Y1}]
There is a projective manifold $X$ with a nef line bundle $L$ which is not semipositive.
\end{theorem}
This counterexample is on a ruled surface $X=\mathbb{P}(E)$ where $E$ is a rank $2$ vector bundle over an elliptic curve $C$ which is a nontrivial extension of $\mathcal{O}_C$ by itself, and $L=\mathcal{O}_{\mathbb{P}(E)}(1)$ is the Serre line bundle. This same counterexample was later rediscovered by Demailly-Peternell-Schneider \cite{DPS}, who also proved that $c_1(L)$ contains a unique closed positive current (see below for definitions). In particular, the inclusions in \eqref{cones} are all strict in general (this is clear for the first inclusion), and the cone $\mathcal{S}$ is neither open nor closed.

In this example we have $\int_X c_1(L)^2=0$, i.e. $L$ is not big. The first example of a nef and big line bundle which is not semipositive was found by Kim \cite[Ex. 2.14]{Kim} on a ruled $3$-fold, and later Boucksom-Eyssidieux-Guedj-Zeriahi \cite[Example 5.4]{BEGZ} observed that one can also find other examples by considering $\mathbb{P}(E\oplus A)$ where $A\to C$ is ample, with its Serre line bundle. The question remained whether one could find an example of a nef and big but not semipositive line bundle on a surface, and in \cite[Problem 2.2]{FT} it was suggested that such an example should exist on a ruled surface constructed by Grauert \cite[8(d), p.365-366]{Gr}. This expectation was realized by Koike \cite{Ko2}, who showed that this is indeed the case:

\begin{theorem}[Koike \cite{Ko2}]
The ruled surface $X$ constructed by Grauert admits a nef and big line bundle $L$ which is not semipositive.
\end{theorem}

\section{Semiample $(1,1)$-classes}

\subsection{Semiample line bundles}

A holomorphic line bundle $L\to X$ over a compact K\"ahler manifold is called semiample if there is $m\geq 1$ such that $L^m$ is globally generated.  In this case, global sections of $L^m$ define a holomorphic map $\Phi:X\to \mathbb{P}^N$ such that $\Phi^*\mathcal{O}(1)=L^m$, and so $\frac{1}{m}\Phi^*\omega_{\rm FS}$ is a smooth semipositive representative of $c_1(L)$. Thus, semiample line bundles are semipositive. The converse implication is false, for example by taking $X$ to be an elliptic curve and $L\in \mathrm{Pic}^0(X)$ a degree-zero non-torsion line bundle. Then $c_1(L)=0$ in $H^{1,1}(X,\mathbb{R})$, so $L$ is trivially semipositive, but $L$ is not semiample.

A more interesting example is given in \cite[Example 10.3.3]{Laz}, of a ruled surface with two nef and big line bundles $L,L'$ with $c_1(L)=c_1(L')$ and $L$ semiample but $L'$ not semiample (so in this case $L'$ is semipositive but not semiample).
What these examples show is that being semiample is not a numerical notion, i.e. it does not only depend on $c_1(L)$. The following definition, which appears in \cite{LP}, is then natural:

\begin{defn}
A line bundle $L\to X$ over a compact K\"ahler manifold $X$ is called numerically semiample if there is a semiample line bundle $L'$ with $c_1(L)=c_1(L')$ in $H^{1,1}(X,\mathbb{R})$.
\end{defn}
From the exponential exact sequence, we see that if $b_1(X)=0$ then every numerically semiample line bundle is actually semiample. Indeed, since $c_1(L\otimes L'^*)=0$ in $H^2(X,\mathbb{R})$, it follows that $c_1(L^k\otimes (L'^*)^k)=0$ in $H^2(X,\mathbb{Z})$ for some $k\geq 1$. Since $X$ is K\"ahler, the assumption $b_1(X)=0$ implies $H^1(X,\mathcal{O})=0$, so the first Chern class map $\mathrm{Pic}(X)\to H^2(X,\mathbb{Z})$ is injective, and so $L\cong L'\otimes M$ where $M$ is a holomorphically torsion line bundle. Since $L'$ is semiample, so is $L'\otimes M$, proving our claim.

We also clearly have that if $L$ is numerically semiample then $L$ is semipositive, and one can ask whether the converse holds. Unfortunately, we have

\begin{theorem}[Koike \cite{Ko1}]
There is a projective manifold $X$ with a semipositive line bundle $L$ which is not numerically semiample.
\end{theorem}
The manifold $X$ and the line bundle $L$ come from a famous example of Zariski: $X$ is the blowup of $\mathbb{P}^2$ along $12$ general points on an elliptic curve $C\subset\mathbb{P}^2$, and the line bundle $L$ is the sum of the pullback of the hyperplane bundle $\mathcal{O}(1)$ on $\mathbb{P}^2$ and the strict transform of $C$. Zariski famously showed that $L$ is not semiample, and since $X$ is simply connected it follows that $L$ is also not numerically semiample, and Koike \cite[Theorem 1.1]{Ko1} shows that $L$ is semipositive.

Nevertheless, despite the fact that (numerical) semiampleness is not the same as semipositivity, showing that a line bundle is semiample remains the easiest way to show that it is semipositive. Thus, theorems that guarantee that a line bundle is semiample are very valuable, and the most prominent such result is the Kawamata-Shokurov base-point-free theorem in birational geometry (which we state here in its simplest form):

\begin{theorem}\label{bpf}
Let $X$ be a projective manifold and $L\to X$ a nef line bundle such that $L^m\otimes K_X^{*}$ is nef and big for some $m\geq 1$. Then $L$ is semiample.
\end{theorem}

In particular, if $c_1(K_X)=0$ in $H^{1,1}(X,\mathbb{R})$ (i.e. $X$ is a projective Calabi-Yau manifold), then every nef and big line bundle is semiample, hence semipositive.

\subsection{Semiample $(1,1)$-classes}
Suppose $L\to X$ is a semiample line bundle. Then if $m\geq 1$ is sufficiently divisible, then global sections of $L^m$ define a holomorphic map $\Phi:X\to\mathbb{P}^N$ with image a normal projective variety $Y$, such that $L^m$ is the pullback of an ample line bundle on $Y$. Passing to a Stein factorization, we may also assume that $\Phi:X\to Y$ has connected fibers.

Inspired by this, in \cite{FT} we posed the following definition:
\begin{defn}
A $(1,1)$-class $[\alpha]$ on a compact K\"ahler manifold $X$ is called  semiample if there is a surjective holomorphic map $f:X\to Y$ with connected fibers onto a normal compact K\"ahler analytic space such that $[\alpha]=f^*[\beta]$ for some K\"ahler class $[\beta]$ on $Y$.
\end{defn}

Clearly, if $L$ is numerically semiample then $c_1(L)$ is semiample, and we can ask about the converse:

\begin{question}\label{bzp} If $L\to X$ is a line bundle over a compact K\"ahler manifold such that $c_1(L)$ is semiample, does it follow that $L$ is numerically semiample?
\end{question}

At the moment, it is not clear to us whether the answer is affirmative, even if we assume that $X$ is projective.  In general, if $c_1(L)$ is semiample then we have a surjective map with connected fibers $f:X\to Y$ with $c_1(L)=f^*[\beta]$, with $[\beta]$ a K\"ahler class on $Y$. We necessarily have $\dim Y\leq \dim X$. Let us make a few observations about Question \ref{bzp}.

If $\dim Y=0$, i.e. $Y$ is a point, then $c_1(L)=0$ in $H^{1,1}(X,\mathbb{R})$, so the answer to Question \ref{bzp} is trivially affirmative in this case, by taking $L'=\mathcal{O}_X$.

If $\dim Y=1$ then Question \ref{bzp} again has an affirmative answer as follows: in this case $Y$ is a smooth compact Riemann surface, so there is an ample line bundle $A\to Y$ such that $[\beta]=\lambda c_1(A)$ for some $\lambda\in\mathbb{R}_{>0}$. Thus
\begin{equation}\label{useless}
c_1(L)=\lambda c_1(f^*A).
\end{equation}
The class $c_1(f^*A)$ lies in $H^2(X,\mathbb{Z})$, and if we denote by $[B]\in H_2(X,\mathbb{Z})$ its Poincar\'e dual, then intersecting \eqref{useless} with $[B]$ gives
$$\lambda = c_1(L)\cdot [B]\in\mathbb{Z},$$
so we see that $\lambda\in\mathbb{N}_{>0}$. Thus,
$$L':=f^*A^{\otimes\lambda},$$
is a semiample line bundle on $X$ with $c_1(L)=c_1(L')$, and so $L$ is numerically semiample.

If $\dim Y=\dim X$ then $f$ is bimeromorphic and $L$ is nef and big, and so $X$ is Moishezon and K\"ahler, hence projective, and $Y$ is a Moishezon space. By a theorem of Nakamaye \cite{Nak}, the augmented base locus $\mathbb{B}_+(L)$ equals the null locus
$$\mathrm{Null}(L)=\bigcup_{\substack{V\subset X\\ (L^{\dim V}\cdot V)=0}}V.$$
From \cite[Prop.2.5]{To} it follows that $$\mathrm{Exc}(f)=\mathbb{B}_+(L)=\mathrm{Null}(L).$$
At this point we would expect that in fact $Y$ is projective, and that there is an ample line bundle $A\to Y$ such that $c_1(A)$ is proportional to $[\beta]$, which (by the argument above) would indeed show that $L$ is numerically semiample.

\subsection{Transcendental base-point-free conjecture}
Motivated by the base-point-free Theorem \ref{bpf}, the author conjectured in \cite[Question 5.5]{To3} a transcendental extension of this to $(1,1)$-classes on Calabi-Yau manifolds, which together with a later conjecture of Filip and the author \cite[Conjecture 1.2]{FT} reads:

\begin{conjecture}\label{c2}
Let $X$ be a compact K\"ahler manifold and $[\alpha]$ a nef $(1,1)$-class such that $\lambda[\alpha]-c_1(K_X)$ is nef and big for some $\lambda\in\mathbb{R}_{>0}$. Then $[\alpha]$ is semiample. In particular, if $c_1(K_X)=0$ in $H^{1,1}(X,\mathbb{R})$ (i.e. $X$ is a Calabi-Yau manifold), then every nef and big $(1,1)$-class is semiample, hence semipositive.
\end{conjecture}

Conjecture \ref{c2} is known for $n\leq 2$ thanks to work of Filip and the author \cite{FT}, and the Calabi-Yau special case is also known when $n=3$ by work of H\"oring \cite{Ho}. The case when $n=3$ of Conjecture \ref{c2} when $\lambda[\alpha]-c_1(K_X)$ is K\"ahler was proved by Zhang and the author \cite{TZ} when $[\alpha]$ is not big and by H\"oring \cite{Ho} when $[\alpha]$ is big. The assumption that $\lambda[\alpha]-c_1(K_X)$ is K\"ahler is actually quite natural, since it is equivalent to $[\alpha]$ being the limiting class of some solution of the K\"ahler-Ricci flow on $X$ with finite time singularity, see \cite{TZ}.  More recently, Conjecture \ref{c2} with $n=3$ was proved in general by Das-Hacon \cite{DH}. Of course, when $X$ is projective and $[\alpha]\in N^1(X)$ then Conjecture \ref{c2} follows from Theorem \ref{bpf}.

In general, work of Collins and the author \cite{CT} shows that the non-K\"ahler locus $E_{nK}([\alpha])$ of a nef and big $(1,1)$-class $[\alpha]$ is a proper closed analytic subvariety of $X$, and when $[\alpha]$ is semiample then $E_{nK}([\alpha])$ should be equal to the exceptional locus of the map $f$ (this holds when $Y$ is smooth, by  \cite[Prop.2.5]{To}, and also in the case of semiample line bundles \cite[Theorem A]{BCL}). Thus, at least in the Calabi-Yau setting of Conjecture \ref{c2}, one would hope to construct $f$ by ``contracting'' $E_{nK}([\alpha])$. In a nutshell, this is what is done in \cite{FT,Ho} in low dimensions, and at this moment the best hope for extending this to all dimensions comes from recent progress in the Minimal Model Program for K\"ahler manifolds by Das, Hacon and coauthors, see e.g. \cite{DH,DH2, DHP}.

\subsection{Generalized abundance conjecture}
In the base-point-free Theorem \ref{bpf}, the line bundle $L^m\otimes K_X^{*}$ is assumed to be nef and big. If the bigness assumption is dropped, Lazi\'c-Peternell \cite{LP} recently proposed the following ``generalized abundance conjecture'', which again state in its simplest form:
\begin{conjecture}[Lazi\'c-Peternell \cite{LP}]\label{lp}
Let $X$ be a projective manifold with $K_X$ pseudoeffective, and $L\to X$ a nef line bundle such that $L^m\otimes K_X^{*}$ is nef for some $m\geq 1$. Then $L$ is numerically semiample.
\end{conjecture}
This is known only when $n\leq 2$ by \cite{LP}.

\subsection{Nef $(1,1)$-classes on Calabi-Yau manifolds}
An important special case the generalized abundance Conjecture \ref{lp}, which was in fact one of the main motivations for it, is the case when $X$ is Calabi-Yau, i.e. $X$ is a compact K\"ahler manifold with $c_1(K_X)=0$ in $H^{1,1}(X,\mathbb{R})$. In this case, the generalized abundance conjecture specializes to the following well-known conjecture:

\begin{conjecture}\label{c3}
If $X$ is a projective Calabi-Yau manifold and $L\to X$ is a nef line bundle, then $L$ is numerically semiample.
\end{conjecture}

This conjecture is a classical result when $n=2$, but it remains open already for $n=3$, see e.g. \cite{LOP} for recent results and an overview.

Going now from line bundles to general $(1,1)$-classes, recall that Conjecture \ref{c2} predicts that every nef and big $(1,1)$-class on a Calabi-Yau manifold is semiample.  One is then naturally led to wonder about the case of nef $(1,1)$-classes that are not big. Taking cue from Conjecture \ref{c3}, one might expect that every nef $(1,1)$-class on a Calabi-Yau manifold must be semiample. This however fails completely:

\begin{theorem}[Filip-T. \cite{FT,FT2}]
There are $K3$ surfaces (both projective and non-projective) with a nef $(1,1)$-class which is not semipositive, hence not semiample.
\end{theorem}

These examples come from holomorphic dynamics, where certain $K3$ surfaces $X$ admit a holomorphic automorphism $T:X\to X$ whose iterates exhibit a chaotic behavior. In this case, Cantat \cite{Can} has constructed two nontrivial nef classes $[\alpha_{\pm}]$, which are not big, which are respectively expanded and contracted by the dynamics, and which contain only one closed positive current $\eta_{\pm}$ with H\"older continuous potentials. A dynamical rigidity theorem, proved by Cantat-Dupont \cite{CD} when $X$ is projective and by Filip and the author \cite{FT2} in general, shows that if $\eta_{\pm}$ is smooth (equivalently, if $[\alpha_{\pm}]$ is semipositive) then $X$ must be a Kummer $K3$ surface (and $T$ must come from an affine transformation of the corresponding torus). There are however plenty of examples of such $(X,T)$ which are not Kummer (including non-projective ones constructed by McMullen \cite{McM}), so on these examples the nef classes $[\alpha_{\pm}]$ are not semipositive.

\section{Currents in nef $(1,1)$-classes}

\subsection{Currents with bounded potentials}
Despite the failure of the direct generalization of Conjecture \ref{c3} to $(1,1)$-classes, the author conjectured in \cite[Conjecture 3.7]{To5} that the following weaker statement should be true:

\begin{conjecture}\label{c4}
If $X$ is a Calabi-Yau manifold and $[\alpha]$ is a nef $(1,1)$-class, then $[\alpha]$ contains a closed positive current with bounded potentials.
\end{conjecture}

Here a closed positive current in the class $[\alpha]$ is a current of the form $\alpha+\ddbar\vp$ which is semipositive in the weak sense, where $\vp$ is a quasi-psh function on $X$ (i.e. an usc $L^1$ function $\vp:X\to \mathbb{R}\cup\{-\infty\}$ which in local charts is the sum of a smooth function and a plurisubharmonic function). If $\vp$ is a bounded function then we say that the current has bounded potentials.

Remarkably, Conjecture \ref{c4} is open even when $n=2$. Even the weaker statement that $[\alpha]$ contains a closed positive current with vanishing Lelong numbers is open (when $[\alpha]$ is not big). In the opposite direction, one could also strengthen the conjecture by requiring that the current have continuous (and not just bounded) potentials. Inspired by the generalized abundance Conjecture \ref{lp}, one can also pose the following more general conjecture, which can  be compared with Conjecture \ref{c2}:

\begin{conjecture}\label{c5}
Let $X$ be a compact K\"ahler manifold with $K_X$ pseudoeffective, and $[\alpha]$ a nef $(1,1)$-class such that $\lambda[\alpha]-c_1(K_X)$ is nef for some $\lambda\in\mathbb{R}_{>0}$. Then $[\alpha]$ contains a closed positive current with bounded potentials.
\end{conjecture}

Let us now discuss in more detail Conjecture \ref{c4} when $n=2$, so $X$ is a Calabi-Yau surface. By the Kodaira classification $X$ is finitely covered by a torus or a $K3$ surface. Since Conjecture \ref{c4} is easily seen to hold on tori, and to behave well under finite covers, we can restrict to the case when $X$ is $K3$, and consider an arbitrary nef $(1,1)$-class $[\alpha]$.  As mentioned earlier, if $[\alpha]$ is also big then $[\alpha]$ is semiample \cite{FT} and so Conjecture \ref{c4} holds in this case. Similarly, if $[\alpha]$ is not big and $\mathbb{R}_{>0}[\alpha]\cap H^2(X,\mathbb{Q})\neq \emptyset$, then $[\alpha]$ is again semiample \cite{FT}.

We are thus left with the case when $[\alpha]$ is nef, not big and ``strongly irrational'' in the sense that $\mathbb{R}_{>0}[\alpha]\cap H^2(X,\mathbb{Q})=\emptyset$. In this case it is expected that $[\alpha]$ contains a unique closed positive current, and this is indeed proved in many cases by Sibony-Soldatenkov-Verbitsky \cite{SSV}, see also \cite[Theorem 4.3.1]{FT3}.
For such strongly irrational classes, the current state of Conjecture \ref{c5} is given by:
\begin{theorem}[Filip-T. \cite{FT3}]
If $X$ is a projective $K3$ surface with no $(-2)$-curves and Picard rank at least $3$, and $[\alpha]$ a nef $(1,1)$-class which is not big, strongly irrational, and $[\alpha]\in N^1(X)\otimes\mathbb{R}\subset H^{1,1}(X,\mathbb{R})$, then $[\alpha]$ contains a unique closed positive current, and this current has continuous potentials.
\end{theorem}
This result uses crucially the ``Kawamata-Morrison Cone Conjecture'' on projective Calabi-Yau manifolds, which is known for $K3$ surfaces \cite{St} but open in dimensions $3$ and higher. New ideas will be needed in order to settle Conjecture \ref{c4} completely when $n=2$.

\subsection{Families of Monge-Amp\`ere equations}
There is a direct link between Conjectures \ref{c2} and \ref{c4} and regularity of solutions of certain complex Monge-Amp\`ere equations. More precisely, let $X^n$ be a compact K\"ahler manifold with a nef $(1,1)$-class $[\alpha]$. For every $\ve>0$ the class $[\alpha]+\ve[\omega]$ is K\"ahler, so by Yau's Theorem \cite{Y} we can solve the family of complex Monge-Amp\`ere equations
$$\omega_\ve:=\alpha+\ve\omega+\ddbar\vp_\ve>0, \quad \omega_\ve^n=c_\ve e^F\omega^n,$$
where $F\in C^\infty(X,\mathbb{R})$ is given, $c_\ve\in\mathbb{R}_{>0}$ is given by
$$c_\ve=\frac{\int_X(\alpha+\ve\omega)^n}{\int_X e^F\omega^n},$$
and $\vp_\ve$ is normalized by $\sup_X\vp_\ve=0$. For example, if $X$ is Calabi-Yau, then with a suitable choice of $F$ the metrics $\omega_\ve$ are Ricci-flat K\"ahler, see the survey \cite{To5} for a detailed discussion of this case.

Going back to our general setting, we have:
\begin{theorem}
In this setting, the nef class $[\alpha]$ contains a closed positive current with bounded potentials if and only if there is $C>0$ such that for all $\ve>0$ small we have
\begin{equation}\label{esti}
\sup_X|\vp_\ve|\leq C.
\end{equation}
\end{theorem}
\begin{proof}
First assume that \eqref{esti} holds. Then from weak compactness of closed positive currents with bounded cohomology classes we see that there is a sequence $\ve_i\to 0$ such that $\vp_{\ve_i}$ converge in $L^1(X)$ to a quasi-psh function $\vp_0$ which satisfies $\alpha+\ddbar\vp_0\geq 0$ weakly, i.e. $\alpha+\ddbar\vp_0$ is a closed positive current in the class $[\alpha]$. From the uniform estimate \eqref{esti} and standard properties of quasi-psh functions, it follows that $\vp_0$ is bounded, as desired.

Now assume conversely that there is a bounded quasi-psh function $\vp_0$ with $\alpha+\ddbar\vp_0\geq 0$, which we can again normalize by $\sup_X\vp_0=0$. Let
$$u_0(x)=\sup\{\eta(x)\ |\ \alpha+\ddbar\eta\geq 0, \eta\leq 0\},$$
be the standard envelope, then we have $\alpha+\ddbar u_0\geq 0$ weakly, and $\vp_0\leq u_0\leq 0$, so $u_0$ is bounded as well. Given $\ve>0$ we consider similarly the envelope
$$u_\ve(x)=\sup\{\eta(x)\ |\ \alpha+\ve\omega+\ddbar\eta\geq 0, \eta\leq 0\},$$
which are easily seen to decrease pointwise to $u_0$ as $\ve\to 0$. Since $u_0$ is bounded, we thus see that there is $C>0$ such that for all $\ve>0$ we have
\begin{equation}\label{kz1}
-C\leq u_0\leq u_\ve\leq 0.
\end{equation}
Then a key result of \cite{BEGZ} (when $[\alpha]$ is big) and \cite{FGS} (when $[\alpha]$ is not big, with a recent new proof in \cite{GPTW}) shows that
there is $C>0$ such that for all $\ve>0$ small we have
$$\sup_X|\vp_\ve-u_\ve|\leq C,$$
and combining this with \eqref{kz1} proves \eqref{esti}.
\end{proof}

Thus, Conjecture \ref{c4} can be restated in terms of a uniform $L^\infty$ estimate for the potentials of Ricci-flat K\"ahler metrics $\omega_\ve$ whose cohomology class converges to $[\alpha]$ as $\ve\to 0$.  From this viewpoint, when $[\alpha]$ is semiample the Ricci-flat metrics $\omega_\ve$ behave very nicely:

\begin{theorem}[T. \cite{To2}, Collins-T. \cite{CT}, Hein-T. \cite{HT}]\label{kk} Let $X$ be a Calabi-Yau manifold, $[\alpha]$ a semiample class and $\omega_\ve=\alpha+\ve\omega+\ddbar\vp_\ve$ the Ricci-flat K\"ahler metric cohomologous to $[\alpha]+\ve[\omega]$. Then there is a closed analytic subvariety $S\subset X$ such that for every $k\in\mathbb{N}$ and every $K\Subset X\backslash S$ there is $C>0$ such that for all $\ve>0$ small we have
$$\|\vp_\ve\|_{C^k(K,\omega)}\leq C.$$
\end{theorem}
Recall that since $[\alpha]$ is semiample, it is of the form $[\alpha]=f^*[\beta]$ for some map $f:X\to Y$ as before and $[\beta]$ K\"ahler on $Y$.
When $[\alpha]$ is semiample and big, Theorem \ref{kk} was shown in \cite{To2,CT}, and $S$ in this case can be taken to be $\mathrm{Exc}(f)=\mathrm{Null}([\alpha])$. The much harder case when $[\alpha]$ is semiample but not big was recently settled in \cite{HT}, and $S$ here can be taken to be the singular fibers of $f$.

\section{The non-K\"ahler case}\label{nk}
In this last section we expand our setting and consider a general compact complex manifold $X^n$, equipped with a Hermitian metric $\omega$, without any K\"ahlerity assumption. The space of $(1,1)$-classes $H^{1,1}(X,\mathbb{R})$ is defined exactly as in \eqref{quot}. When $X$ is non-K\"ahler (i.e. it does not admit any K\"ahler metric) then its K\"ahler cone $\mathcal{C}\subset H^{1,1}(X,\mathbb{R})$ is empty by definition. However, following \cite{DPS}, one can still meaningfully define a cone $\mathcal{N}\subset H^{1,1}(X,\mathbb{R})$ of nef $(1,1)$-classes (not equal to the closure $\mathcal{C}$ in general), by imitating the characterization in \eqref{succa}, namely given a closed real $(1,1)$-form $\alpha$ on $X$ we will say that $[\alpha]\in\mathcal{N}$ whenever given any $\ve>0$ there is $\vp_\ve\in C^\infty(X,\mathbb{R})$ such that \eqref{succa} holds on $X$.

A very basic question that we raised in 2016 (see \cite[(3.2)]{To}) is the following:
\begin{conjecture}\label{c0}
Let $[\alpha]\in\mathcal{N}$, and $V\subset X$ be a closed irreducible analytic subvariety of dimension $k>0$. Then we have
$$\int_V\alpha^k\geq 0.$$
\end{conjecture}

It is not hard to see that Conjecture \ref{c0} is equivalent to the following:
\begin{conjecture}\label{c1}
Let $[\alpha]\in\mathcal{N}$. Then we have
$$\int_X\alpha^n\geq 0.$$
\end{conjecture}
\begin{proof}[Proof of the equivalence of Conjectures \ref{c0} and \ref{c1}]
It is clear that Conjecture \ref{c0} implies Conjecture \ref{c1}. Conversely, assume Conjecture \ref{c1} and let $V^k\subset X$ be a closed irreducible subvariety. By Hironaka's resolution of singularities of analytic spaces, we can find a modification $\mu:\ti{X}\to X$ with $\ti{X}$ a compact complex manifold, such that $\mu$ is an isomorphism over the generic point of $V$ and the strict transform $\ti{V}$ of $V$ via $\mu$ is smooth. Then $\ti{V}$ is also $k$-dimensional and the class $[\mu^*\alpha]$ is nef on $\ti{X}$: indeed, given a Hermitian metric $\ti{\omega}$ on $\ti{X}$, there is a constant $C>0$ such that
$$\mu^*\omega\leq C\ti{\omega},$$
on $\ti{X}$. Given $\ve>0$ let $\vp_\ve\in C^\infty(X,\mathbb{R})$ be a smooth function such that
$$\alpha+\ddbar\vp_\ve\geq -\ve\omega,$$
on $X$. Pulling back via $\mu$ we get
$$\mu^*\alpha+\ddbar(\vp_\ve\circ\mu)\geq -\ve\mu^*\omega\geq -C\ve\ti{\omega},$$
and so $[\mu^*\alpha]$ is nef. Since $\mu|_{\ti{V}}$ is generically an isomorphism with its image, it follows that
$$\int_V\alpha^k=\int_{\ti{V}}\mu^*\alpha^k.$$
Now the restricted class $[\mu^*\alpha|_{\ti{V}}]\in H^{1,1}(\ti{V},\mathbb{R})$ is immediately seen to be nef, and so by Conjecture \ref{c1} we have
$$\int_{\ti{V}}\mu^*\alpha^k\geq 0,$$
as desired.
\end{proof}

We observe also that Conjecture \ref{c1} is an immediate consequence of a stronger conjecture proposed by Ko\l odziej and the author in \cite[Conjecture 1.2]{KT}, which is the conjectural formula
$$\int_X\alpha^n=\inf_{u\in C^\infty(X,\mathbb{R})}\int_{X(\alpha+\ddbar u,0)}(\alpha+\ddbar u)^n,$$
where $X(\alpha+\ddbar u,0)\subset X$ is the subset of points $x\in X$ where $(\alpha+\ddbar u)(x)\geq 0$, since on this set we clearly have $(\alpha+\ddbar u)^n\geq 0$.

We have the following easy partial results towards Conjecture \ref{c1}:
\begin{proposition}
Conjecture \ref{c1} holds in all of the following cases:
\begin{itemize}
\item[(a)] If $X$ admits a Hermitian metric $\omega$ with $\de\db\omega=0=\de\db(\omega^2)$. In particular, if $X$ is K\"ahler.
\item[(b)] If $X$ is in Fujiki's class $\mathcal{C}$. In particular, if $[\alpha]$ is big, in the sense that it contains a K\"ahler current.
\item[(c)] If $\dim X\leq 2$.
\item[(d)] If $[\alpha]$ contains a closed positive current $\alpha+\ddbar\vp\geq 0$ with $\vp$ bounded.
\item[(e)] If $X$ admits a Hermitian metric $\omega$ with
$$v_+(\omega):=\sup\left\{\int_X (\omega+\ddbar\vp)^n\ \bigg|\ \vp\in {\rm PSH}(X,\omega)\cap L^\infty(X)\right\}<\infty.$$
\end{itemize}
\end{proposition}
The quantity $v_+(\omega)$ in (e) was recently introduced by Guedj-Lu \cite{GL}.
\begin{proof}
(a) Let $\vp_\ve$ be smooth functions such that $\alpha+\ddbar\vp_\ve\geq -\ve\omega$ on $X$.
A well-known direct calculation shows that we have $\de\db(\omega^k)=0$ for all $1\leq k\leq n-1$. We can then integrate by parts as usual and obtain that
$$0\leq \int_X(\alpha+\ve\omega+\ddbar\vp_\ve)^n=\int_X(\alpha+\ve\omega)^n,$$
and letting $\ve\to 0$ concludes the proof.

(b) $X$ in class $\mathcal{C}$ means that we can find a modification $\mu:\ti{X}\to X$ with $\ti{X}$ a compact K\"ahler manifold. As shown earlier, the pullback class $[\mu^*\alpha]$ is nef and $$\int_X\alpha^n=\int_{\ti{X}}\mu^*\alpha^n.$$
Since $\ti{X}$ is K\"ahler, we conclude using part (a). It is also known \cite{DP} that if $[\alpha]$ is big then $X$ is in class $\mathcal{C}$.

(c) The case when $\dim X=1$ is elementary. If $\dim X=2$ then by a classical theorem of Gauduchon \cite{Ga}, $X$ admits a Hermitian metric $\omega$ with $\de\db\omega=0$, so part (a) applies.

(d) By Demailly's regularization theorem \cite{Dem2} applied to the bounded function $\vp$, we can find smooth functions $\vp_\ve$ on $X$ which decrease pointwise to $\vp$ as $\ve\to 0$, and such that
$$\alpha+\ddbar\vp_\ve\geq -\ve\omega.$$
Since $\vp$ is bounded, it follows that $$\|\vp_\ve\|_{L^\infty(X)}\leq C,$$
for all $\ve>0$. Then a standard Chern-Levine-Nirenberg type argument in \cite[Claim, p.997]{KT} shows that
\begin{equation}\label{useless2}
0\leq \lim_{\ve\to 0}\int_X(\alpha+\ve\omega+\ddbar\vp_\ve)^n=\int_X\alpha^n,
\end{equation}
as desired.

(e) For $\ve>0$ choose smooth functions $\vp_\ve$ as in (a), then the proof of \cite[Theorem 4.12]{GL} shows that the assumption $v_+(\omega)<\infty$ implies that \eqref{useless2} again holds.
\end{proof}

Despite these partial results, Conjecture \ref{c1} remains open even when $n=3$. The special case when $[\alpha]=c_1(L)$ for a holomorphic line bundle $L$ might be more approachable, for example using Demailly's holomorphic Morse inequalities \cite{Dem}, but so far it remains open as well.  In fact, despite the fact that failure of Conjecture \ref{c1} would be rather strange indeed, the author suspects that perhaps counterexamples do exist, and it might be desirable to do some explicit computations in the hope of finding one.

\end{document}